\documentclass[final,12p,times,reqno]{amsart}

\usepackage[colorlinks=true,linkcolor=black, citecolor=blue, urlcolor=blue]{hyperref}
\usepackage{amsfonts,latexsym}
\usepackage{mdwlist}  %%% suspend and resume enumerate
\usepackage{amssymb}
\usepackage{amsmath, times}

\author[N. ATTIA]% \hfil ....-..../..\hfilneg]
{ Najmeddine ATTIA }  % in alphabetical order

 \bigskip  \bigskip
\address{\newline Department of Mathematics, Faculty of
Sciences of Monastir, University of Monastir, 5000-Monastir,
Tunisia}

\email{najmeddine.attia@gmail.com}

\newcommand{\R}{\mathbb R}

\newcommand{\N}{\mathbb N}

\newcommand{\HH}{{\mathsf H}}

\newtheorem{theorem}{Theorem}
\newtheorem{lemma}{Lemma}
\newtheorem{proposition}{Proposition}
\newtheorem{corollary}{Corollary}
\newtheorem{definition}{Definition}
\newtheorem{remark}{Remark}

\usepackage{rotating}

\newcommand{\rle}{\rotatebox[origin=c]{90}{$\le$}}

 \numberwithin{equation}{section}

\title[  On the Hewitt Stromberg dimension of product sets]{ On the Hewitt Stromberg dimension of product sets}

\begin{document}

\maketitle

\begin{abstract}
 In this paper, we  construct    new multifractal measures, on the Euclidean space  $\R^n$,  in a similar manner to Hewitt-Stomberg meausres
  but using the class of all $n$-dimensional half-open binary cubes
  of covering sets in the definition rather than the class of all balls.
   As an application we shall be concerned with evaluation of Hewitt-Stromberg dimension of
    cartesian product sets by means of the dimensions of their components.\\

\bigskip

\noindent{Keywords}: Multifractal measures, Hewitt-Stromberg
measures, product sets.

\bigskip
\noindent{Mathematics Subject Classification:} 28A78, 28A80.
\end{abstract}

\maketitle

\section{Introduction }

Hewitt-Stromberg measures were introduced  in \cite[Exercise
(10.51)]{HeSt}. Since then, they  have  been investigated  by
several authors, highlighting their importance in the study of local
properties of fractals and products of fractals. One can cite, for
example \cite{Ha1, Ha2, Baek1, HHBaek, Olsen18}. In
particular, Edgar's textbook \cite[pp. 32-36]{Ed} provides an
excellent and systematic introduction to these  measures. Such
measures  also appears explicitly, for example, in  Pesin's
monograph \cite[5.3]{Pe} and implicitly in Mattila's text
\cite{Mat}. The reader can be referred  to  \cite{Olsen18,
Olsen19,attia19a, attia19b, Attia21a, Attia21b, Attia20, Attia23, Attia21c} for a class of generalization of these
measures). The aim of this paper is to construct a metric outer
measure $\HH^{*t}$ comparable with the Hewitt-Stromberg measure
$\HH^t$ (see Proposition \ref{comparaison}). In the construction of
these measures we use
 the class of all $n$-dimensional half-open binary cubes for covering
sets  rather than the class  of all balls (see
Section \ref{construction}). As an application, we discuss and prove
in Section \ref{application} the relationship between Hewitt-Stromberg
dimension of cartesian product sets and  the dimensions of their
components. We obtain in particular,
$$
\dim_{MB} (A\times B) \ge  \dim_{MB}  A +  \dim_{MB} B,
$$
for a class of subsets of $\R$, where $\dim_{MB} $ denote the Hewitt-Stomberg dimension. Various
results on this problem have been obtained for Hausdorff and packing
dimension (see for example \cite{Bes45, Marstrand54,Ohtsuka57,Xiao96,howroyd96,wei16, Attia21, Attia22b, Attia20, Attia22c}).
We give in the end of  section  \ref{application} a sufficient condition to get
the equality in the previous equation (Theorem \ref{egalite}). In
the Section \ref{example} we construct two sets $A$ and $B$ such
that $\dim_{MB} (A\times B) \neq  \dim_{MB} A + \dim_{MB} B.$ Which
proves that the last  inequality can be strict.
%This is not surprising under the similar result on Hausdorff and packing dimensions in the papers cited above.

 \section{Preliminary}
 First we recall briefly the definitions of Hausdorff dimension,
 packing  dimension and Hewitt-Stromberg dimension and the relationship linking these three
 notions.  Let  ${\mathcal F} $ be the class of dimension functions, i.e.,
 the  functions $h : \R_+^* \to \R_+^*$   which are right continuous, monotone increasing  with $\lim_{r\to 0}h(0)=0$.

Suppose that,  for $n\ge 1$, $\R^n$  is  endowed with the Euclidean distance.
 For  $E\subset \R^n$, $h\in {\mathcal F}$ and $\varepsilon>0$,  we write
$$
\mathcal{H}_\varepsilon^h(E)= \inf\left\{\sum_i
h\Big(|E_i|\Big)\;\; E\subseteq\bigcup_i
E_i,\;\; |E_i|<\varepsilon \right\},
$$
where $|A|$ is the diameter of the set $A $ defined as
$|A| = \sup\big\{ |x - y|, \; x, y \in A \big\}$. This
allows to define the  Hausdorff measure, with respect to $h$,
 of $E$ by
$$
\mathcal{H}^h(E)=\sup_{\varepsilon>0}\mathcal{H}_\varepsilon^h(E).
$$
The reader can be  referred to Rogers' classical text \cite{Rogers} for a systematic discussion of $\mathcal{H}^h$.

%For $x\in \R^m$,  and $r>0$, let $B(x, r)$  denote the closed ball with enter $x$ and radius $r$.
We define, for $\varepsilon> 0$, 
\begin{eqnarray*}
\overline{\mathcal{P}}_\varepsilon^h(E)= \sup\left\{\sum_i
h\Big(2r_i\Big)\right\},
\end{eqnarray*}
where the supremum is taken over all closed balls $\Big( B(x_i, r_i)
\Big)_i \; \text{such that}\; r_i \leq \varepsilon$,   $x_i \in E \; \text{and}\; |x_i -  x_j|\geq \frac{r_i+r_j}2\;
\text{for}\; i\neq j$. The $h$-dimensional packing premeasure, with respect to $h$,
 of $E$ is now defined by
$$
\overline{\mathcal{P}}^h(E)=\sup_{\varepsilon>0}\overline{\mathcal{P}}_\varepsilon^h(E).
$$
This makes us able to define the  packing measure, with respect to $h$,
 of $E$ as
$$
{\mathcal{P}}^h(E)=\inf\left\{\sum_i\overline{\mathcal{P}}^h(E_i)\;\Big|\;\;E\subseteq\bigcup_i
E_i\right\}.
$$
% \bigskip

While Hausdorff and packing measures are defined using coverings and
packings by families of sets with diameters less than a given
positive number $\varepsilon$, the Hewitt-Stromberg measures are
defined using covering of balls with the same diameter
$\varepsilon$. The Hewitt-Stromberg premeasure
$\overline{\mathsf{H}}^h$ is defined by
$$
\overline{\mathsf{H}}^h(E)=\liminf_{r\to0}
\overline{\mathsf{H}}^h_{r} \quad \text{where} \quad
\overline{\mathsf{H}}^h_r (E) =  N_r(E) \;h(2r)
$$
and the  covering number $N_r(E)$ of $E$ is  defined by
\begin{eqnarray*}
N_r(E)=\inf\Big\{\sharp\{I\}\;&\Big |&\; \Big( B(x_i, r) \Big)_{i\in
I} \;  \text{is a family of closed balls } \\
&& \text{ with}\; x_i \in E \; \text{and}\; E\subseteq \bigcup _i
B(x_i, r)\Big\}.
\end{eqnarray*}
Now, we define  the    Hewitt-Stromberg
measure, with respect to $h$,   which we denote   by $\mathcal{\mathsf H}^h$, as follows
$$
{\mathsf{H}}^h(E)=\inf\left\{\sum_i\overline{\mathsf{H}}^h(E_i)\;\Big|\;\;E\subseteq\bigcup_i
E_i\right\}.
$$
\begin{remark}
In a similar manner  to Hausdorff and packing measures, for $E\subseteq \R^n$ and $t\ge 0$, we have
$$\overline{\HH}^t(E) = \overline{\HH}^t(\overline{E}),$$
where $\overline{E}$ is the closure of $E$.
\end{remark}
We recall the basic inequalities satisfied by the Hewitt-Stromberg,
the Hausdorff  and the packing measures (see \cite[Proposition
2.1]{Olsen18})

$$\begin{array}{ccccccc}
 & & \overline{\mathsf{H}}^h(E) & \le & \overline{\mathcal{P}}^h(E)\\
& & \rle  & & \rle\\
{\mathcal{H}}^h(E) & \leq & {\mathsf{H}}^h(E) & \leq & {\mathcal{P}}^h(E).
\end{array}
$$
 Let $t>0$ and $h_t$ is the dimension function defined by $$h_t(r) =
r^t.$$
 In this case we will denote simply ${\mathcal H}^{h_t}$ by
${\mathcal H}^t$, also ${\mathcal P}^{h_t}$ will be denoted by $
{\mathcal P}^t$, $\overline{\HH}^{h_t}$ will be denoted by $\overline{\HH}^t$ and $\HH^{h_t}$ will be denoted by $\HH^t$.  Now we
define the Hausdorff dimension, the packing dimension and the
Hewitt-Stromberg dimension of a set $E$ respectively by
$$
\dim_H E =\sup \left\{ t \ge 0, \;\; {\mathcal H}^t (E)= +\infty
\right\} = \inf \left\{ t \ge 0, \;\; {\mathcal H}^t (E)= 0
\right\},$$
$$ \dim_P E =\sup \left\{ t \ge 0, \;\; {\mathcal P }^t (E)= +\infty,
\right\} =\inf \left\{ t \ge 0, \;\; {\mathcal P }^t (E)= 0\right\}
$$
and
$$ \dim_{MB} E =\sup \left\{ t \ge 0, \;\; {\HH}^t (E)= +\infty
\right\} = \inf \left\{ t \ge 0, \;\; {\HH}^t (E)= 0 \right\}.$$
%$$
% \overline{\dim}_{MB} E =\sup \left\{ t \ge 0, \;\; \overline{\HH}^t = +\infty
%\right\}.
%$$
It follows, for any set $E$, that
 $$ \dim_H (E) \le \dim_{MB} (E) \le \dim_P (E). $$
\begin{definition}
Let $\xi >0$. A set $E$ is said to be $\xi$-regular if, for any $t \ge 0$, we have
$$
\overline{\HH}^t(E) =\xi \HH^t(E).
$$
That is,   $E$ is $\xi$-regular if  ${\dim}_{\overline {MB}} E = {\dim}_{MB} E =\alpha$ and $\overline{\HH}^\alpha(E)  =\xi \HH^\alpha(E)$, where
$${\dim}_{\overline {MB}} E=
\sup \left\{ t \ge 0, \;\; \overline{\HH}^t (E)= +\infty \right\} = \inf \left\{ t \ge 0, \;\; \overline{\HH}^t(E) = 0 \right\}.$$
\end{definition}

 We  finish this section by two lemmas which
will be useful in the following.
\begin{lemma}\label{ball}
Let $B$ is a ball in $\R^n$ of diameter $\delta >0$. The
number of balls of diameter $\gamma \in (0, \delta)$ necessary to
cover $B $ is less then
$$
b_n := \Big[\frac{\delta}{\gamma} \sqrt{n} \Big]^n.
$$
\end{lemma}
\begin{proof}
Consider a ball $B$ of diameter $\delta$. $B$ can be inscribed in a cube of side length $\delta$. In the other hand the largest cube that can be inscribed in a ball of diameter $\gamma$ has diameter $\gamma$ and therefore has side $\displaystyle\frac{\gamma}{\sqrt n}$. Thus, we need $$\frac{\delta}{\gamma} \sqrt n$$ edges of the smaller cubes to completely cover an edge of the largest cube, and hence we would need $b_n$ of the smaller cubes to cover the largest cube, thereby also covering the ball of diameter $\delta$. Since each ball of diameter $\gamma$ contains one of these smaller cubes, we can therefore use this number of balls to cover the ball of diameter $\delta$.
\end{proof}
\begin{remark}\label{covering} As a direct application of Lemma \ref{ball}, if $k$ is an integer, any cube of side $2^{-k}$  is contained in $(2n)^n$ balls of diameter $2^{-k-1}$.
\end{remark}
\begin{lemma}\label{2r}
Let $\{E_n\} $ be a decreasing sequence of compact subsets of\;
$\R^n$ and $F = \bigcap_n E_n$. Then, for any $\delta >0$, $t\ge 0$ and $\gamma>1$,
$$
\lim_{n\to +\infty}\overline{\HH}^t_{\gamma\delta} (E_n) \le \gamma^t \overline{\HH}^t_\delta(F).
$$
\end{lemma}
\begin{proof}
Let $\Big\{B_i = B(x_i, \delta )\Big\}$ be  any covering of $F$. We claim that there exists $n$ such that
 $E_n \subset U =\bigcup_i B(x_i, \gamma\delta)$. Indeed, otherwise, $\Big\{ E_n \backslash U \Big\}$ is a
 decreasing sequence of non-empty compact sets, which, by an elementary consequence of compactness,
 has a non-empty limit set $(\lim E_n )\backslash U$. Then, for $t\ge 0$,
$$
\lim_{n\to +\infty} N_{\gamma\delta} (E_n) (2\gamma\delta)^t\le \gamma^t N_{\delta}(F) (2\delta)^t.
$$
\end{proof}

%As a consequence we get the following result

%%%%%%%%%%%%%%%%%%%%%%%%%%%%%%%%%%%%%%%%%%%%%%%%%%%%%%%%%%%%%%%%%%%%%%%%%%%%%%%%%%%%%%%%%%%%%%

\section{Relation between $\HH^t $ and $\overline{\HH}^t$}
%%%%%%%%%%%%%%%%%%%%%%%%%%%%%%%%%%%%
%%%%%%%%%%%%%%%%%%%%%%%%%%%%%%%%%%%%
 We  can see, from the definition,  that estimating the Hewitt-Stromberg premeasure  is much easier than estimating the Hewitt-Sttromberg measure. It is therefore natural to look
for relationships between these two quantities. The reader can also see \cite{Joy95, feng99, Wen07, attia18} for a similar result for Hausdorff and packing measures.

 \begin{lemma}
 Let $K$ be compact set in $\R^n$ and $t\ge 0$. Suppose that for   every $\epsilon >0$ and  subset $E$ of $K$ one can find an open set $U$ such that $E\subset U$ and  $\overline{\HH}^t(U \cap K) \le \overline{\HH}^t(E)+\epsilon$, then
$$\HH^t(K) = \overline{\HH}^t(K).$$
 \end{lemma}
 \begin{proof}
 Let $\epsilon >0$ and let $\{E_i\}$ be a sequence of sets  such that    $K\subseteq \bigcup_i E_i$. 
 %Therefore,
% we can find, for each $i$,   $r_i >0$ such that
% $$
% \overline{\HH}^t_{r_i}(E_i) \le \overline{\HH}^t(E_i) + 2^{-i-1} \epsilon.
% $$
 Take, for each $i$, a set $U_i$ such that  $E_i\subset U_i$ and  $$\overline{\HH}^t(U_i\cap K) \le \overline{\HH}^t(E_i)+ 2^{-i-1}\epsilon.$$
  Since $K$  is compact, the cover $\{U_i \}$ of $K$ has a finite
subcover. So we may use the fact that, for all $F_1, F_2\subset
\R^n$, 
 $$
 \overline{\HH}^t (F_1\cup F_2)\le \overline{\HH}^t (F_2) \cup\overline{\HH}^t (F_2)
 $$
 to infer that
 $$
 \overline{\HH}^t(K) \le \sum_{i} \overline{\HH}^t(U_i\cap K) \le \sum_i (\overline{\HH}^t(E_i)+ 2^{-i-1}\epsilon) \le  \sum_i \overline{\HH}^t(E_i) + \epsilon.
 $$
 This is true for all $\epsilon >0$ and $\{E_i\}$ such that $K\subseteq \bigcup_i E_i$. Thus $$
\HH^t(K) \ge \overline{\HH}^t(K).$$
The opposite inequality is obvious.
 \end{proof}

\begin{theorem}
Let $K\subset \R^n$ be a compact set and $ t\ge 0$  such that
$\overline{\HH}^t(K) <+\infty. $ Then for
 any subset $F$ of $K$ and any $\epsilon >0$ there exists an open set $U$ such that $F\subset U$ and
$$
\overline{\HH}^t(U \cap K) < \overline{\HH}^t(F) + \epsilon.
$$
\end{theorem}
\begin{proof}
Since $F$ has the same Hewitt-Stromberg premeasure as its closure we
can assume that $F$ is  a compact set. For $n\ge 1$, define the
$n$-parallel body $F_n$ of $F$ by $$ F_n = \Big\{ x \in \R^n, \quad
|x-y| < 1/n, \;\; \text{for some} \; y\in F\Big\}.
$$
It is  clear that $F_n$ is  an open set  and $F\subset F_n$, for all $n$. Denote by $\overline{F}_n$ the closure of $F_n$ and let $\gamma >1.$ Using Lemma \ref{2r}, there exists $n$ such that
$$
\overline{\HH}^t(\overline{F}_n \cap K) \le \gamma^t
\overline{\HH}^t(F)
$$
For $\epsilon >0$, we can choose $\gamma $ such that $\gamma^t
\overline{\HH}^t(F) \le \overline{\HH}^t(F)  + \epsilon$. Finally,
we get
$$
\overline{\HH}^t(F_n\cap K) \le  \overline{\HH}^t(\overline{F}_n
\cap K)  \le   \overline{\HH}^t(F)  + \epsilon.
$$
\end{proof}
As a  direct consequence, we get the following results.
\begin{theorem}\label{egaH}
Let $K\subset \R^n$ be a compact set and $ t\ge 0$. If \;
$\overline{\HH}^t(K) <+\infty$  then
 $$\overline{\HH}^t(K)= {\HH}^t(K).$$
\end{theorem}
From Theorem  \ref{egaH}, we immediately obtain the following corollary.
\begin{corollary} Let $E\subset \R^n$ and $t\ge 0$
\begin{enumerate}
\item Assume that $0 < \overline{\HH}^t(E) < +\infty$. Then $0 < {\HH}^t (\overline{E})< \infty$.
\item Assume that $E$  is compact and $t > \dim_{MB} E$. Then either $\overline{\HH}^t (E) = 0$
 or $\overline{\HH}^t(E) =+\infty$.\\
\end{enumerate}
\end{corollary}
The following corollary shows that the theorems of Besicovitch \cite{Bes52} and Davies \cite{Davies52} for Hausdorff measures and the theorem of Joyce and Preiss \cite{Joy95} for packing measures does not hold for the Hewitt-Stromberg premeasure.
\begin{corollary}
There exists a compact set $K$ and $t > 0$ with $\overline{\HH}^t(K)
=+\infty$  such that $K$ contains no subset with positive finite
Hewitt-Stromberg premeasure.
\end{corollary}
\begin{proof}
Consider for $n\ge 1$, the set  $A_n = \{ 0\}\bigcup \{ 1/k, \;\;
k\le n\}$ and $$K= \bigcup_n A_n = \Big\{ 0 \Big\}\; \bigcup \;
\Big\{ 1/n, \;\; n\in \N \Big\}.$$
Now, we will prove that ${\dim}_{\overline {MB}} K = 1/2$.
%$$\overline{\dim}_{MB} K=
%\sup \left\{ t \ge 0, \;\; \overline{\HH}^t (K)= +\infty \right\} = \inf \left\{ t \ge 0, \;\; \overline{\HH}^t(K) = 0 \right\}.$$
 For $n\ge 1$ and $\delta_n = \frac{1}{n+n^2}$, remark that  $$ N_{\delta_n}(A_n) = n+1.$$
 It follows that
 $$
 \overline{\HH}_{\delta_n}^{1/2}(K)\ge
 \overline{\HH}_{\delta_n}^{1/2}(A_n) = \sqrt{2} \frac{n+1}{\sqrt{n+n^2}}.
 $$
 Thereby,  $ \overline{\HH}^{1/2}(K)>0$ which implies that $ { \dim}_{\overline {MB}} K \ge 1/2$. In the
  other hand,  if $\overline{\dim}_p(K)$ denote the box-counting dimension of
   $K$, i.e.,
 $$
\overline{\dim}_p(K) =\sup\{t ;\;\overline{\mathcal P}^t(K)=
+\infty\} =\inf\{t ;\;\overline{\mathcal P}^t(K)= 0\}
 $$
 then $\overline{\dim}_p(K) =\frac1 2$ (see Corollary 2.5 in  \cite{feng99}) and thus
 $${ \dim}_{\overline {MB}} K \le \overline{\dim}_p(K) =1/2.$$ As a
 consequence, we have ${ \dim}_{\overline {MB}} K = 1/2.$ Take $t=1/3$, it is cleat that
 $\HH^t(K) = 0$. Moreover,  $ \overline{\HH}^t(K)=+\infty$. It follows,  for any subset $F
$ of $K$, that $\overline{\HH}^t(F) =0$ or $+\infty$. Otherwise, assume
that $0< \overline{\HH}^t(F) <+\infty$. Then $0<
\overline{\HH}^t(\overline{F}) <+\infty$ and thus, by using  Theorem
\ref{egaH}, \; $0< {\HH}^t(F) <+\infty$, which is impossible since
$F$ is a subset of $K$.
\end{proof}

%%%%%%%%%%%%%%%%%%%%%%%%%%%%%%%%%%%%%%%%%%%%%%%%%%%%%%%%%%%%%%%%%%%%%%%%%%%%%%%%%%%%%%%%%%%%%%%

\section{Construction of the  multifractal
measures}\label{construction}

%%%%%%%%%%%%%%%%%%%%%%%%%%%%%%%%%%%%%%%%%%%%%%%%%%%%%%%%%%%%%%%%%%%%%%%%%%%%%%%%%%%%%%%%%%%%%%%%
%%%%%%%%%%%%%%%%%%%%%%%%%%%%%%%%%%%%%%%%%%%%%%%%%%%%%%%%%%%%%%%%%%%%%%%%%%%%%%%%%%%%%%%%%%%%%%%%

In a similar way to Hewitt-Stromberg measure $\HH^t$  we will
construct a new measure  $\HH^{*t}$ but using a restricted class
${\mathcal A}$ of covering set. We prove that $\HH^t$ and
$\HH^*$ are indeed comparable measures which is very useful tool in
the study of Hewitt-Stromberg measure. Let ${\mathcal A}$ be the
collection of all $n-$dimensional half-open binary cubes, i.e., the
collection ${\mathcal C}_k^n$ of cubes
$$
C= I_1\times\cdots \times I_n,
$$
where each $I_i \subset \R$ is an interval of the form $I_i = [u_i, v_i)$ with $u_i = p_i 2^{-k}, v_i = (p_i +1) 2^{-k}, p_i $ is an integer and $k$ is a non-negative integer. If $n=1$ or $2$, then these cubes are certain intervals or squares. Let $E\subset \R^n$ and $k$ be non negative integer. We define
 the  covering number $N^*_{2^{-k}}(E)$ of $E$ to be the infimum number of the family of binary cubes of side  $2^{-k}$ that cover the set  $E$. For $t \ge0$, we define
 % where $k$ stand for the unique integer such that $$2^{-(k+1)}<  2r \le 2^{-k}.$$ Now, let us define
$$
\overline{\mathsf{H}}^{*t}_{2^{-k}}(E)= N^*_{2^{-k}}(E) \;2^{-kt}\qquad
\text{and}\qquad  \overline{\mathsf{H}}^{*t}(E) =\liminf_{k\to +\infty}\; \overline{\mathsf{H}}^{*t}_{2^{-k}}(E).  $$
The function $\overline{\mathsf{H^*}}^t$ is increasing  but not $\sigma$-subadditive. That is the reason for which we will introduce the following modification to define a measure
$$
{\mathsf{H}^*}^t(E)=\inf\left\{\sum_i\overline{\mathsf{H}^*}^t(E_i)\;\Big|\;\;E\subseteq\bigcup_i
E_i\right\}.
$$
%and
%$$
%{\mathsf{P}}^h(E)=\inf\left\{\sum_i\overline{\mathsf{P}}^h(E_i)\;\Big|\;\;E\subseteq\bigcup_i
%E_i\right\}.
%$$
\begin{proposition}
$\HH^{*t}$ is a metric outer measure on \, $\R^n$ and thus measure
on the Borel family of subsets of \; $\R^n$.
\end{proposition}
\begin{proof}
 Let $E,F \subset\mathbb{R}^n$ such that $d(E,F) = \inf\left\{ |x-y|, x\in E, y\in F \right\} >0.$ Since
${\HH}^{*t}$ is an outer measure,  it suffices to prove that
$$\mathsf{H^*}^t\Big(E\bigcup
F\Big)\geq\mathsf{H^*}^t(E)+ \mathsf{H^*}^t(F).$$
Let $k$ be an integer such that
$$0 < 2^{-k}\sqrt{n} < d(E,F)/2.$$
 Consider $\{C_i\}$  a  familiy of binary cubes of side $2^{-k}$ that cover  $E\bigcup F$. Put
  $$I=\Big\{i;\: C_i\bigcap E\neq\emptyset\Big\} \qquad \text{and}\qquad J=\Big\{i;\: C_i \bigcap  F\neq\emptyset\Big\}.$$ It is clear that $\{ C_i\}_{i\in I}$ cover $E$ and $\{ C_i\}_{i\in J}$ cover $F$. It  follows  that
\begin{center}
$\displaystyle{N}^{*}_{2^{-k}}\Big(E\bigcup
F\Big)\geq\displaystyle{N}^{*}_{2^{-k}}(E)+\displaystyle{N}^{*}_{2^{-k}}(F)$\end{center}
and then
$$
\overline{\HH}^{*t}\Big(E\bigcup
F\Big)\geq\overline{\HH}^{*t}(E)+\overline{\HH}^{*t}(F).
$$
This implies that
\begin{eqnarray*}
\HH^{*} \Big(E\bigcup F\Big) &=& \inf_{E\cup F
\subseteq \bigcup_{i}E_i}\left\{ \sum_i {\overline\HH}^{*t}(E_i) \;\right\}\\
&\geq&\inf_{E\cup F \subseteq \bigcup_{i}E_i}\left\{ \sum_i
{\overline\HH}^{*t}(E_i\cap E)+\sum_i {\overline\HH}^{*t}(E_i\cap F) ;\right\}\\
&\geq& \inf_{E\cup F \subseteq \bigcup_{i}E_i}\left\{ \sum_i
{\overline\HH}^{*t}(E_i\cap E) \right\}+\inf_{E\cup F \subseteq\bigcup_{i}E_i}\left\{\sum_i {\overline\HH}^{*t}(E_i\cap F)\;\right\}.
\end{eqnarray*}
Finally, we conclude that
$$
\HH^{*} \Big(E\bigcup F\Big)  \geq\HH^*(E)+\HH^{*}(F).
$$

\end{proof}
\begin{proposition}\label{comparaison}
For every set $E\subset \R^n$, we have, for any $t\ge 0$,
\begin{equation}\label{relationH}
b_n^{-1}{\mathsf{H}}^t(E) \le {\HH^*}^t(E) \le \alpha_n {\mathsf{H}}^t(E),
\end{equation}
where  $\alpha_n =3^n$ and $b_n = (2n)^n$.
\end{proposition}

\begin{proof}
  Let $\Big( B_i = B(x_i,  2^{-k-1}) \Big)_{i\in I} $ is a family of closed balls  with $x_i \in E \; \text{and}\; E\subseteq \bigcup _i B_i$.
Each $B_i$ is contained in the collection of $\alpha_n = 3^n$ binary cubes of side $2^{-k}$ and its immediate neighbours.  Therefore,
% Each of these cubes   can be Subdividing into  $2^{n^2}$ smaller cubes, then $B_i$ is contained in $\alpha_n = 3^n 2^{n^2}$ binary cubes of side $2^{-n-k}$. Thus $B_i \subseteq \bigcup_{j=1}^{\alpha_n} C_{i,j}$, where $C_{i,j}$ is a collection of $\alpha_n$ cubes of side $2^{-k-n}$. Therefore $E\subseteq \bigcup_i \bigcup_j C_{i,j}$ and
$$
N_{2^{-k}}^*(E) \le \alpha_n N_{2^{-k-1}}(E).
$$
It follows, for $t\ge 0$, that
 \begin{equation*}\label{NN*}
N_{2^{-k}}^*(E) 2^{-kt} \le \alpha_n N_{2^{-k-1}}(E) 2^{-kt}
\end{equation*}
and then, by letting $k \to +\infty$,
\begin{equation}\label{HH*}
{\overline{\mathsf{H}}^*}^t(E) \le \alpha_n
{\overline{\mathsf{H}}}^t(E). \end{equation}
Now suppose that $E \subseteq \bigcup E_i$, then
$$
\HH^{*t} (E ) \le \sum_i {\overline{\mathsf{H}}^*}^t(E_i)
\le\alpha_n \sum_i {\overline{\mathsf{H}}}^t(E_i).$$
Since $\{E_i\}$ is an arbitrarily covering of $E$ we get the right-hand inequality of \eqref{relationH}.\\
Conversely,  each cube $C_i$ of side $2^{-k}$ which  intersect $E$
is contained, by Remark \ref{covering},  in a $b_n = (2n)^n$   balls
with diameter $2^{-k-1}$. Therefore  $C_i$ is contained in  $(2n)^n$
balls whose centers belongs to $E$ with diameter $ 2^{-k} $. Thus, for $t\ge 0$,
we have
$$
N_{2^{-k-1}} (E)  2^{-kt}\le b_n N^*_{2^{-k}}(E) 2^{-kt}.
$$
Letting $k\to +\infty$, we obtain $${\overline{\mathsf{H}}}^t(E) \le
b_n {\overline{\mathsf{H}}}^{*t}(E).$$ Now suppose that $E \subseteq
\bigcup E_i$ then
$$
\HH^{t} (E ) \le \sum_i {\overline{\mathsf{H}}}^t(E_i)  \le b_n
\sum_i {\overline{\mathsf{H}}}^{*t}(E_i).$$ Since $\{E_i\}$ is an
arbitrarily covering of $E$, we get the left-hand inequality of
\eqref{relationH}.
\end{proof}

%\begin{lemma}
%Let $r>0$ and let  $\{E_j\}$  be a decreasing sequence of compact subsets of $\R^n$. If $E=\lim_{j\to +\infty} E_j$, then,
%$$
%N_r ( E) > \lim_{j\to +\infty} N_{2r}(E_j).
%$$
%\end{lemma}

%%%%%%%%%%%%%%%%%%%%%%%%%%%%%%%%%%%%%%%%%%%%%%%%%%%%%%%%%%%%%%%%%%%%%%%%%%%%%%%%%%%%%%%%%%%%%%%%%%%%%%%%%%%%%%%%%%%%%
%%%%%%%%%%%%%%%%%%%%%%%%%%%%%%%%%%%%%%%%%%%%%%%%%%%%%%%%%%%%%%%%%%%%%%%%%%%%%%%%%%%%%%%%%%%%%%%%%%%%%%%%%%%%%%%%%%%%%
%%%%%%%%%%%%%%%%%%%%%%%%%%%%%%%%%%%%%%%%%%%%%%%%%%%%%%%%%%%%%%%%%%%%%%%%%%%%%%%%%%%%%%%%%%%%%%%%%%%%%%%%%%%%%%%%%%%%%
\section{Application : Cartesian products of
sets}\label{application}

%%%%%%%%%%%%%%%%%%%%%%%%%%%%%%%%%%%%%%%%%%%%%%%%%%%%%%%%%%%%%%%%%%%%%%%%%%%%%%%%%%%%%%%%%%%%%%%%%%%%%%%%%%%%%%%%%%%%%
%%%%%%%%%%%%%%%%%%%%%%%%%%%%%%%%%%%%%%%%%%%%%%%%%%%%%%%%%%%%%%%%%%%%%%%%%%%%%%%%%%%%%%%%%%%%%%%%%%%%%%%%%%%%%%%%%%%%%
%%%%%%%%%%%%%%%%%%%%%%%%%%%%%%%%%%%%%%%%%%%%%%%%%%%%%%%%%%%%%%%%%%%%%%%%%%%%%%%%%%%%%%%%%%%%%%%%%%%%%%%%%%%%%%%%%%%%%
In this section, for simplicity,   we restrict the result to subsets
of the plane, though the work extends to higher dimensions without
difficulty.  Given a plane set  $E\subset \R^2$, we denote by   $E_x$
the set of its  points whose abscisse are equal to $x$.
\begin{theorem}\label{product}
Consider a plane set  $F$   and let $ A$ be any subset of the x-axis. Suppose that, if  $x\in A$, we have $\HH^t(F_x) >c$, for some constant $c$. Then
$$\overline{\HH}^{s+t} (F) \ge \gamma c \HH^s(A),$$
where $\gamma =    b_1^{-2}  \alpha_1^{-1}$.
\end{theorem}

\begin{proof}
Let   $k$ be a non negative integer and $\{C_i\}$ be a collection of
binary squares of side $2^{-k}$ covering $F$.
%It's clear that, for
%$x\in A$,  we have
%$$F_x \subset \bigcup_i (C_i)_x.
%$$
%Let, for each $i$, $\widetilde C_i$ be the binary interval of length $2^{-k}$  obtained by projection $C_i$ onto the $x-$axix. Thus
% $$
%N^*_{2^{-k}}(E_x) \le  \# \big\{ i ; \; \; x\in \widetilde C_i \big\}.
%$$
Now, put 
$$
A_k =  \big\{ x\in A, \;\; N_{2^{-k}}^* (F_x) 2^{-kt}
> b_1^{-1} c \Big\}.
$$
Remark that $ \# \Big\{ C_i \Big\} \ge N_{2^{-k}}^*(A_k)
\inf\Big\{N_{2^{-k}}^*(F_x), \; x\in A_k\Big\}. $ Therefore,
$$\# \Big\{ C_i \Big\} 2^{-k(s+t)} \ge b_1^{-1} c N_{2^{-k}}^* (A_k) 2^{-ks}.$$
But this is true for any covering of $F$ by binary squares $\{C_i\}$
with side $2^{-k}$, so
$$
 b_1^{-1} c \overline{\HH}_{2^{-k}}^{*s}(A_k) \le \overline{\HH}^{*t+s}_{2^{-k}}(F) \le  \overline{\HH}^{*t+s}(F).
$$
Since $A_k$ increase to $A$ as $k \to +\infty$, then for any $p\le k$ we have 
$$
 b_1^{-1} c  \overline{\HH}^{*s}_{2^{-k}}(A_{p}) \le  b_1^{-1} c  \overline{\HH}^{*s}_{2^{-k}}(A_{k}) \le  \overline{\HH}^{*t+s}(F).
$$
 Thus, using \eqref{HH*}, we obtain
$$
  b_1^{-1} c \HH^{*s}(A_p)\le b_1^{-1} c  \overline{\HH}^{*s}(A_p) \le   \overline{\HH}^{*t+s}(F)\le \alpha_1 \overline{\HH}^{s+t}(F),$$
for $p\ge 1$. Thereby, the continuity of the measure $\HH^*$ implies that
$$
b_1^{-1} c \HH^{*s}(A) \le \alpha_1 \overline{\HH}^{s+t}(F).
$$
Thus, using Proposition \ref{comparaison}, we get
  $$
 b_1^{-2}  c \HH^{s}(A) \le b_1^{-1} c \HH^{*s}(A)  \le
 \alpha_1\overline{ \HH}^{s+t}(F).
 $$
%where $\xi =\overline{\HH}^{s+t}(F) / \HH^{s+t}(F) $.
Finally by taking $\gamma =    b_1^{-2}  \alpha_1^{-1} $, we get the
 result.
 \end{proof}

\begin{corollary} Under the same conditions of Theorem \ref{product}. If in addition, $F$ is a $\xi$-regular set then
$$\HH^{s+t} (F) \ge \gamma \xi^{-1} c \HH^s(A).$$
In particular if $F=A\times B$, where $A, B\subset\R$, then
 \begin{equation}\label{mesure-times}
 \HH^{s+t}(A\times B) \ge  \gamma \xi^{-1} \HH^s(A) \HH^t(B)
 \end{equation}
 and thus
 \begin{equation} \label{dim-times}
   \dim_{MB}(A\times B) \ge \dim_{MB} A + \dim_{MB} B.
   \end{equation}
\end{corollary}
%\begin{remark}

%Let $A, B$ two bounded subsets of \; $\R$ such that $\overline{\HH}^{s+t}(A\times B)< +\infty$.  Since $\overline{A \times B}$ is compact then is $1-$regular. Using Theorem \ref{egaH}.
%\begin{eqnarray*}
%  \HH^{s+t}(A\times B) = \overline{\HH}^{s+t}(\overline{A\times B}) & \ge&  \gamma \HH^s(\overline {A}) \HH^s(\overline{B})= \HH^s(A) \HH^s(B).
% \end{eqnarray*}
%\end{remark}
We can construct  two sets $A$ and $B$ such that $ \dim_{MB}(A\times B)
> \dim_{MB} A + \dim_{MB} B$ (see the next section). Then, it is
interesting to know if there is some sufficient condition to get the
 equality in \eqref{dim-times}. For this, for $t\ge 0$,  we define the lower
$t$-dimensional density  of a set $E$ at $y$ by
$$
d^t(y) = \liminf_{h\to 0}\frac{\HH^t\Big(E\cap B(y,h)\Big)}{(2h)^s}.
$$
\begin{theorem}\label{egalite}
Let $A$ be a set of point in $x$-axis such that $0< \HH^s(A)<+\infty$ and 
let $B$ a set of point in $y$-axis
 such that $0< \HH^t(B)<+\infty.$ Suppose that \eqref{dim-times} is satisfied and, for
all $y\in B$, $d^t (y) >0$ then
$$
\dim_{MB} (A\times B) = \dim_{MB}(A) + \dim_{MB}(B).
$$
\end{theorem}
\begin{proof} Define, for $h>0$, the set $I_y(h)$ to be the centered interval on $y$ with length  $h$. For $n \ge 1$, consider the set
$$
B_n = \left\{ y\in  B, \;\; \HH^t \Big(B\cap I_y(h)\Big ) >
h^t/n, \quad \forall h \le n^{-1} \right\}.
$$
Under the hypothesis   $d^t(y) >0$ for all $y\in B$ we have
clearly that $B_n \nearrow B$. Suppose that we have shown  that
there exists $n\in \N$ such that
\begin{equation}\label{produit}
\overline{\HH}^{s+t} (A\times B_n) <+\infty.
\end{equation}
Then, it follows at once that $\dim_{MB} A\times B = s+t$. \\

Let us prove \eqref{produit}. Let $n$ be an integer and $0 < h \le 1/n$.
Define
$$
I(h) = \big\{ I_y(h),\quad y\in B_n \big\}.
$$
We can extract from $I(h)$  a finite subset $J(h)$ such that $B_n
\subset J(h)$ and no three intervals of $J(h)$ have points in
common. Now divide the set $J(h)$ into $J_1(h)$ and $J_2(h)$ such
that in each of which the intervals do not overlap. Therefore, the
cardinal of the sets $J_1(h)$ and $J_2(h)$ is less than $nh^{-t}
\HH^t(B)$. Indeed, using the defintion of the set $B_n$, we get
$$
h^{-t}n \HH^t(B) \ge \sum_{I\in J_1(h)} h^{-t}n
\HH^t(B \cap I) >  \#  J_1(h).
$$
Thus $\# J(h) \le 2 nh^{-t} \HH^t(B)$.\\
For $\epsilon >0$, there exists a sequence of sets $\{A_i\}$ such
that
$$
\sum_i \overline{\HH}_h^s(A_i) \le \sum_i \overline{\HH}^s(A_i) \le \HH^{s}(A)  +\epsilon.
$$
%Thereby, if $n$ big enough and $0<h\le 1/n$,
%$$
%\sum_i \overline{\HH}_h^s(A_i) \le \HH^{s}(A)  +\epsilon.
%$$
Thereby, there exists a sequence of intervals  $\{U_{i,
j}\}$ of length $h$ covering $A$ such that for each $i$, we have
$\{U_{i, j}\}$ is a $h$-cover of $A_i$ and
$$
\#\big\{U_{i, j} \big\} h^s  \le \HH^{s}(A)  +\epsilon.
$$
 Let $[a, b]$ be any interval of  $\{U_{i, j}\}$.
  Enclose all the  points of the set $A\times B_n$ lying between
   tine $x=a$ and $x=b$ in the set of squares, with sides on
    these lines, whose projections on the $y$-axis are the intervals of $J(h)$. Also,
    construct a similar sets of squares corresponding to each interval of  $\{U_{i, j}\}$  and denote the sets of squares  corresponding to the interval $[a, b]$ by $C(a, b)$. Since $\# C(a, b)$ does
     not exceed $\# J(h)$ and each square can be inscribed in a ball of
     diameter $h'=\sqrt{2} h$, we obtain
$$
N_{h'/2}(A\times B_n) \le \# J(h)\;  \# \{U_{i, j}\}.
$$
Thus
\begin{eqnarray*}
\overline{\HH}^{s+t}_{h'/2}(A\times B_n) &\le& 2 nh^{-t} \HH^t(B) (\sqrt{2} h)^{s+t}    \# \{U_{i, j}\}\\
&\le &  2^{\frac{1}{2} (s+t +2)}  n \HH^t(B)    \sum_{i, j} h^s\\
&\le &  2^{\frac{1}{2} (s+t +2)}  n \HH^t(B)
(\HH^s(A)+\epsilon),
\end{eqnarray*}
from which the equation \eqref{produit} follows.
\end{proof}

%%%%%%%%%%%%%%%%%%%%%%%%%%%%%%%%%%%%%%%%%%%%%%%%%%%%%%%%%%%%%%%%%%%%%%%%%%%%%%%%%
%%%%%%%%%%%%%%%%%%%%%%%%%%%%%%%%%%%%%%%%%%%%%%%%%%%%%%%%%%%%%%%%%%%%%%%%%%%%%%%%%
\section{Example} \label{example}
%%%%%%%%%%%%%%%%%%%%%%%%%%%%%%%%%%%%%%%%%%%%%%%%%%%%%%%%%%%%%%%%%%%%%%%%%%%%%%%%%
%%%%%%%%%%%%%%%%%%%%%%%%%%%%%%%%%%%%%%%%%%%%%%%%%%%%%%%%%%%%%%%%%%%%%%%%%%%%%%%%%

In general the inequalities in \eqref{dim-times} and
\eqref{mesure-times} may be strict. In this section, we will
construct two sets $A$ and $B$ such that
 $$\dim_{MB} A + \dim_{MB} B < \dim_{MB} (A\times B).$$ Before construction of these sets
we give the following useful lemma.
\begin{lemma} \label{projection}
Let $\psi : E\subset \R^2 \to F\subset \R$ be a surjective mapping such that, for $x, y \in E$,
$$
|\psi(x) - \psi(y) | \le c |x-y|,
$$
for a constant $c$. Then, for $t \ge 0$, $$ \HH^t(F) \le c^t
\HH^t(E).$$
\end{lemma}
\begin{proof}
Let $E_i \subset E$  and $F_i$ be the set  such that $\psi(E_i) = F_i$.  It is clear that for any covering  of $E_i$ by a balls with radius $\delta$ we can construct a covering of $F_i$ by a balls with radius $(c\delta)$. Therefore, for $t\ge 0$,
$$
N_{c\delta}(F_i) (2c\delta)^t \le c^t N_{\delta}(E_i) (2 \delta)^t.
$$
Thus
$$
\overline{\HH}^t(F_i)  \le c^t \overline{\HH}^t(E_i).
$$
Now, if $E\subset \bigcup_i E_i$ with  $E_i\subset E$ and let $\{F_i\}$ be the sets such that $\psi(E_i) =F_i$. Then
$$
\HH^t(F) \le \sum_{i}\overline{\HH}^t(F_i) \le c^t \sum_i \overline{\HH}^t(E_i).
$$
Since $\{E_i\}$ is an arbitrarily covering  of $E$ we get the result.
\end{proof}
  Let $\{t_j\}$ be a decreasing sequence of numbers with $\displaystyle\lim_{j\to+\infty} t_j = 0$ and let $\{m_j\}$ be a increasing sequence of integers. We can Choose $m_0= 0$ and $\{m_j\}_{j\ge 1}$ rapidly enough to ensue that, for all $j\ge 1$,
\begin{equation}\label{contrainte}
\sum_{k=0}^{j-1} m_{2k+1} - m_{2k} \le t_j m_{2j} \qquad \text{and}\qquad \sum_{k=1}^{j} m_{2k} - m_{2k-1} \le t_j m_{2j+1}.
\end{equation}
Consider the set $A\subset [0,1]$ such that, if $r$ is odd and $m_j
+1\le r\le m_{j+1}$ then the $r$-th decimal place is zero, i.e., $A$
is the set of $x$ such that
$$
x= 0,x_1\ldots x_{m_1}\underbrace{0 \;\; \ldots \ldots \;\; 0}_{(m_2-m_1) times}x_{m_2+1}\ldots x_{m_3}\underbrace{0\;\; \ldots \ldots \;\; 0}_{(m_4-m_3) times }\ldots
$$
where $x_i\in \{0, 1, \ldots, 9\}$. Similarly take the set $B\subset
[0,1]$ such that, if $r$ is even and $m_j +1\le r\le m_{j+1}$ then
the $r$th decimal place is zero, i.e.,   $B$ is the set of $x$ such
that
$$
x= 0,\underbrace{0 \;\; \ldots\ldots \;\; 0}_{m_1 times} \; x_{m_1+1}\ldots x_{m_2} \underbrace{0 \;\; \ldots\ldots \;\; 0}_{(m_3-m_2) times} \; x_{m_3+1}\ldots x_{m_4}\ldots
$$
where $x_i\in \{0, 1, \ldots, 9\}$. It is clear that we can cover $A$
by $10^k$ intervals of length $10^{-m_{2j}}$ where
$$
k  = (m_1-m_0) + (m_3-m_2)+\cdots + (m_{2j-1} - m_{2j-2}),
$$
it follows from \eqref{contrainte} that, if $t>0$ then
$$\HH^t(A) \le \overline{\HH}^t(A) =0.$$
As a consequence, we prove $\dim_{MB} A =0$ and similarly we have
$\dim_{MB} B =0$. Now let $\psi$ denote orthogonal projection from
the plane onto the line $L : y = x$. Then $\psi(x, y)$  is the point
of $L$ at distance
$$\sqrt{2}(x + y)$$ from the origin. Take $u\in [0,1]$ we may find
two number $x \in A$ and $y \in B$ such that $u = x + y$, indeed
 some of the decimal digits of $u$ are provided by $x$, the rest by $y$. Thus $\psi(A\times B)$ is a subinterval
 of $L$ of length $\sqrt{2}$. Using the fact that orthogonal projection does not increase distances and so, by  Lemma \ref{projection}, does not
 increase Hewitt-Stromberg  measures,
 \begin{eqnarray*}
\HH^1(A\times B) &\ge&  \HH^1\Big(\psi(A\times B)\Big) \ge    {\mathcal
H}^1\Big(\psi(A\times B)\Big) \\
&=&  {\mathcal L}\Big(\psi(A\times B)\Big) =\sqrt{2}.
 \end{eqnarray*}
 where ${\mathcal L}$ is the Lebesgue measure on $\R$. This imply
 that $$\dim_{MB}(A\times B) \ge 1 > \dim_{MB}(A)+\dim_{MB}(B).$$

\end{document}